\documentclass[11pt]{amsart}
\usepackage[centering]{geometry}                
\usepackage{graphicx}
\usepackage{xcolor}
\usepackage{url}
\usepackage{enumerate}
\usepackage{amssymb}
\usepackage{amsrefs}
\usepackage{epstopdf}
\usepackage{tikz}
\usetikzlibrary{shapes,patterns,calc,snakes}
 
\usepackage[foot]{amsaddr}

\usepackage{mathtools}

\newcommand\mbb{\mathbb}

\newcommand\mfr{\mathfrak}

\newcommand\fG{\mfr{G}}

\newcommand\N{\mbb{N}}

\DeclareMathOperator{\cpsdrank}{cpsd-rank}
\DeclareMathOperator{\cprank}{cp-rank}
\DeclareMathOperator{\diam}{diam}
\DeclareMathOperator{\Diag}{Diag}
\DeclareMathOperator{\tr}{Tr}

\theoremstyle{plain}
\newcounter{count}
\newtheorem{Thm}[count]{Theorem}

\newtheorem{Cor}[count]{Corollary}
\newtheorem{Lemma}[count]{Lemma}

\newtheorem*{Thm*}{Theorem}
\newtheorem*{Prop*}{Proposition}
\newtheorem*{Cor*}{Corollary}
\newtheorem*{Lemma*}{Lemma}
\newtheorem*{Sublemma*}{Sublemma}
\newtheorem*{Conjecture*}{Conjecture}

\theoremstyle{definition}

\newtheorem{Def}[count]{Definition}

\newtheorem{Example}[count]{Example}

\newtheorem{Remark}[count]{Remark}

\newtheorem*{Construction*}{Construction}
\newtheorem*{Def*}{Definition}
\newtheorem*{Defs*}{Definitions}
\newtheorem*{Example*}{Example}
\newtheorem*{Examples*}{Examples}
\newtheorem*{LemmaDef*}{Lemma and Definition}
\newtheorem*{Notation*}{Notation}
\newtheorem*{Problem*}{Problem}
\newtheorem*{Question*}{Question}
\newtheorem*{Remark*}{Remark}
\newtheorem*{Remarks*}{Remarks}
\newtheorem*{Warning*}{Warning}

\usepackage[textsize=footnotesize]{todonotes}

\title{Approximate Completely Positive Semidefinite Factorizations and Their Ranks}
\author{Paria Abbasi}
\address{University of Innsbruck, Department of Mathematics, Austria}
\author{Andreas Klingler}
\address{University of Innsbruck, Institute for Theoretical Physics, Austria}
\author{Tim Netzer}
\address{University of Innsbruck, Department of Mathematics, Austria}

\date{\today}                                           

\begin{document}
\maketitle

\begin{abstract}In this paper we show the existence of approximate \textit{completely positive semidefinite (cpsd)} factorizations with a cpsd-rank bounded above (almost) independently from the cpsd-rank of the initial matrix. This is particularly relevant since  the cpsd-rank of a matrix cannot, in general, be upper bounded by a function only depending on its size.

For this purpose, we make use of the Approximate Carath\'eodory Theorem in order to construct an approximate matrix with a low-rank  Gram representation. We then employ the Johnson-Lindenstrauss Lemma to improve to a logarithmic dependence of the  cpsd-rank on the size.

\end{abstract}

\section{Introduction}
The scalar product of two vectors in the nonnegative orthant $\mathbb R_+^d$ is always nonnegative. If $n$ nonnegative vectors $v_1,\ldots, v_n$ are given, the (entrywise nonnegative) matrix $$M=\left(\langle v_i,v_j\rangle\right)_{i,j=1,\ldots, n}$$ is called a {\it completely positive matrix}. So the possible combinations of inner products between tuples of nonnegative vectors are encoded in the convex cone of all  completely positive matrices.  
The cone of completely positive matrices has numerous applications in control theory and general optimization, among others.  It has been intensively studied, see for example \cite{ber}. 

Replacing nonnegative vectors by positive semidefinite matrices $A_1, \ldots, A_n$ and the inner product by the trace inner product, one obtains a non-commutative analogue $$M = \left(\tr(A_i A_j)\right)_{i,j = 1, \ldots, n},$$ called a \emph{completely positive semidefinite matrix}. Completely positive semidefinite matrices include completely positive ones since the nonnegative vectors of the latter decomposition embed as diagonal psd matrices in the former decomposition. We call the cone of all matrices attaining such a factorization the cone of completely positive semidefinite matrices. These  cones allow for a conic optimization approach to quantum correlations and quantum graph colorings, for example \cite{lau, pra, si}. Note that the trace inner product of two psd matrices is always nonnegative, so both completely positive and completely positive semidefinite matrices have nonnegative entries.

The non-commutative/quantum analogue is more challenging to understand than the commutative/classical version. For example, the cone of completely positive matrices is closed, whereas the cone of completely positive semidefinite matrices has recently been shown not to be closed if the matrix size is $n \geq 10$ \cite{sl, dy, pra}. Also, to  both kinds of decompositions there is an associated rank,  measuring the minimal size of the nonnegative vectors/positive semidefinite matrices that are necessary to represent the given matrix. Whereas the {\it completely positive rank} can be bounded in terms of the size $n$ alone, this fails for the {\it completely positive semidefinite rank}. This significant difference serves as the main motivation for this paper. 

In this work, we show that every completely positive semidefinite matrix attains an approximation of completely positive semidefinite matrices of relatively small completely positive semidefinite rank. This rank of the approximation depends on the size $n$, the accuracy of the approximation, and a certain complexity of the initial matrix. However, most importantly, it does not depend on its completely positive semidefinite rank. We provide two such results, one being better for fixed approximation error and $n$ very large, the other better for fixed $n$ and  small approximation error.

The main ingredients of our proof are the {\it Approximate Carath\'eodory Theorem} and the {\it Johnson-Lindenstrauss Lemma}. We will first approximate the initial matrix, using the Approximate Carath\'eodory Theorem. This first step will already establish the first upper bound. In a second step, we will then further improve the approximation  by applying the Johnson-Lindenstrauss Lemma to the eigendecomposition of the positive semidefinite matrices in the representation. This will reduce the linear dependence on $n$ of the upper bound  to a logarithmic dependence, and so establish the second upper bound. We only show the existence of an approximate decomposition and do not explicitly construct an approximation, since the Johnson-Lindenstrauss Lemma follows from a probabilistic existence argument.

Section \ref{Sec:Not} contains the essential preliminary material and reviews known results about the non-closedness of the cone of completely positive semidefinite matrices. Section \ref{Sec:Main} then contains our main result and some examples.

\section{Notations and preliminaries}\label{Sec:Not}
We will first state some basic definitions and results used throughout this paper. 
Let $[n]$ be the set $\{1, \ldots, n\}$ and $\mathcal{S}^n$ be the space of all $n \times n$ real symmetric matrices (i.e.\ $A^{t} = A$) endowed with the trace inner product: $$\langle A, B\rangle=\text{tr}(BA)=\sum_{i,j=1}^{n}A_{ij}B_{ij}.$$ The corresponding norm is known as the $\text{Frobenius}$ norm: $\Vert A \Vert_{\mathcal{F}}=\sqrt{\langle A, A\rangle}.$ 

A non-empty subset $\mathcal{C} \subseteq \mathcal{S}^n$ is called a \textit{convex cone} if it is closed under nonnegative linear combinations, i.e.\ for all $ \alpha, \alpha^\prime \geqslant 0$ and for all $ c, c^\prime \in \mathcal{C}$ we have $\alpha c+\alpha^\prime c^\prime \in \mathcal{C}.$ We call $\mathcal{C}$ {\it pointed}, if $\mathcal{C} \cap -\mathcal{C} = \{0\}$; we call it {\it full-dimensional} if it has non-empty interior and we call it {\it closed}, if it is a closed set in the Euclidean topology. We call a convex cone with these three properties a {\it proper} cone.
Given a convex cone $\mathcal{C} \subseteq \mathcal{S}^n,$ its dual cone is defined as $$\mathcal{C}^* \coloneqq \{A \in \mathcal{S}^n: \langle A, B\rangle \geqslant 0 \quad \textrm{for all } B \in \mathcal{C}\}$$ and this is always  a closed convex cone. 

A real symmetric matrix $A \in \mathcal{S}^n$ is called \textit{positive semidefinite} ({\it psd} for short, and denoted by $A \succcurlyeq0$)  if there exist vectors $v_1,\ldots, v_n\in \mathbb{R}^d, $ for some $d \in \mathbb{N},$ such that $A=(\langle v_i, v_j \rangle)_{i,j=1}^{n}$, where $\langle\cdot,\cdot\rangle$ denotes the standard inner product on $\mathbb R^d$. We also say that the vectors $v_1,\ldots, v_n$ form a \textit{Gram representation} of $A.$ It is well-known that the smallest possible $d$ in a Gram representation of $A$ coincides with its usual matrix rank ${\rm rank}(A)$.
We denote by $\mathcal{S}_+^{n}$ the set of all $n \times n$ positive semidefinite matrices and it is well-known that it is a proper and self-dual cone, i.e., $\mathcal{S}_+^n=(\mathcal{S}_+^{n})^*.$ 

A real symmetric matrix $A \in \mathcal{S}^n$ is called \textit{doubly nonnegative} if it is both positive semidefinite and entrywise nonnegative. So this means that it  admits a Gram representation by vectors $v_1,\ldots, v_n\in \mathbb{R}^d,$  for which the pairwise angles between the $v_i$ do not exceed $\pi/2,$ i.e.\ $\langle v_i, v_j\rangle \geqslant0$ for all $i, j \in [n].$  The set of all $n \times n$ doubly nonnegative matrices is known to form a proper cone, which is denoted by $\mathcal{DNN}^n.$ 

A real symmetric matrix $A \in \mathcal{S}^n$ which has a Gram representation by  entrywise nonnegative vectors $v_1,\ldots, v_n\in \mathbb{R}^d_{+},$ for some $d \in \N,$ is called \textit{completely positive} ({\it cp} for short). The smallest possible such $d$ is known as the {\it cp-rank } of $A$. The best known upper bound on the  cp-rank of an $n\times n$ cp-matrix is  $\frac{(n+1)n}{2} - 1$ (see \cite{sb2}) while it is not known if this upper bound is tight (see \cite[Conjecture 3.1]{ber}). The difference of this upper bound to the best known lower bound has recently been improved to $\mathcal O(n\log\log n)$, see \cite{sch}. The set of all $n \times n$ completely positive matrices also forms a proper cone, denoted by $\mathcal{CP}^n$.   The structure of the cone $\mathcal{CP}^n$ has been extensively studied (see for example \cite{ber}). 

For a nonnegative vector $x_i$ we denote by $D_i = \mathrm{Diag}(x_i)$ the diagonal psd matrix whose diagonal entries are the entries of $x_i$. We then obtain
$$(\langle x_i, x_j\rangle)_{i,j=1}^{n}=(\langle D_i, D_j\rangle)_{i,j=1}^{n},$$
where we use the trace inner product on the right hand side. In particular, every $D_i$ is a positive semidefinite matrix. At this point, it is natural to pass from diagonal psd matrices to general psd matrices. In this way, we obtain the cone of $\textit{completely positive semidefinite}$ matrices.
\begin{Def}
A matrix $A \in \mathcal{S}^n$ is called \textit{completely positive semidefinite} ({\it cpsd} for short), if it admits a Gram representation by psd matrices $A_1,\ldots, A_n \in  \mathcal{S}^d_+$, for some $d \geqslant 1,$ meaning $$A=\left(\langle A_i,A_j\rangle\right)_{i,j=1}^n.$$ We denote the set of all $n \times n$ completely positive semidefinite matrices by $\mathcal{CPSD}^n$.  
\end{Def} 

Note that the set $\mathcal{CPSD}^n$ forms a convex cone, similarly to $\mathcal{CP}^n$ \cite{dy}.

By definition, every completely positive matrix is also completely positive semidefinite. Moreover, since the trace inner product of two psd matrices is nonnegative and every psd matrix in the Gram representation of a cpsd matrix can be considered as a vector (by stacking the columns on top of each other), we obtain the following inclusions:
\begin{equation}
\mathcal{CP}^n \subseteq \mathcal{CPSD}^n \subseteq \mathcal{DNN}^n.
\end{equation}
In addition, the cone $\mathcal{CPSD}^n$ is pointed and full-dimensional, which directly follows from the facts that the cone $\mathcal{DNN}^n$ is pointed and the cone $\mathcal{CP}^n$ is full-dimensional, respectively. One of the challenging questions about the cone $\mathcal{CPSD}^n$ is whether it is closed. For $n \leqslant 4,$ it is known that $\mathcal{CP}^n=\mathcal{CPSD}^n=\mathcal{DNN}^n$ (first proven in \cite{max}), and hence $\mathcal{CPSD}^n$ is closed. On the other hand, for $n \geqslant 5$ both inclusions above are strict (see \cite{faw} and \cite{fre}). Furthermore, by characterization of the closure of $\mathcal{CPSD}^n$-cone given in \cite{bu}, together with the example in \cite{fre}, the chain of inclusions above can even be refined to (also shown in \cite{lau})
	$$\mathcal{CP}^n \subsetneq \mathcal{CPSD}^n \subseteq \text{cl}(\mathcal{CPSD}^n) \subsetneq \mathcal{DNN}^n.$$
A recent breakthrough is \cite{sl}, where it is shown that a certain affine section of the completely positive semidefinite cone is not closed, and hence the same holds for the cone $\mathcal{CPSD}^n$ itself, for  $n \geqslant 1942$. The lower bound on $n$ was improved in \cite{dy}, where it was subsequently shown that $\mathcal{CPSD}^n$ is not closed even for $n \geqslant 10.$ Hence, it remains an open problem whether the cone is closed for $n \in \{5, 6, 7, 8, 9\}.$

As the $\mathcal{CPSD}^n$-cone is a generalization of the $\mathcal{CP}^n$-cone, it is also natural to extend the notion of rank in the latter cone by replacing the nonnegative vectors with psd matrices.
\begin{Def}
The {\it completely positive semidefinite rank} of a matrix $A \in \mathcal{CPSD}^n,$ denoted by {\rm cpsd-rank}($A$), is the smallest  $d \geqslant 1$ for which there exist psd matrices $A_1, \ldots, A_n \in \mathcal{S}^d_+$  such that 
$$A = (\langle A_i,A_j\rangle )_{i,j=1}^{n}.$$
\end{Def}

\begin{Remark} Instead of using real symmetric psd matrices $A_1,\ldots, A_n\in \mathcal S^d_+$ in the Gram representation of cpsd-matrices, one can also use complex Hermitian matrices $A_1,\ldots, A_n\in\mathcal H^d_+$. This gives rise to the same notion of cpsd-matrices, only decreases the cpsd-rank by a factor of at most two. This can be seen by using the isometry
$$ \mathcal{H}^{d} \longrightarrow \mathcal{S}^{2d}; \quad M\longmapsto \frac{1}{\sqrt{2}}\begin{pmatrix}
\text{Re}(M)&&-\text{Im}(M)\\
\text{Im}(M)&&\text{Re}(M)
\end{pmatrix}$$
which preserves positive semidefiniteness. We will restrict to the case of real symmetric psd matrices from now on. \end{Remark}

In the following, we introduce another notion of rank for cpsd-matrices, which will be used in the proof of our main result.

\begin{Def}
Let $A \in \mathcal{CPSD}^n.$ We define the {\it Gram-cpsd-rank} of $A$ (denoted cpsd-rank$_{\fG}(A)$) as the smallest $r \geqslant 1$ for which there exists a Gram representation $A_1,\ldots, A_n \in \mathcal{S}_+^d,$ for some $d \in \mathbb{N},$ with rank$(A_i)\leqslant r$ for all $i\in[n]$. 
\end{Def} 

The next lemma shows the relationship between the two notions of cpsd-ranks. The result is similar to Lemma 2.1 in \cite{san} and Lemma 3.1 in \cite{pra}, we include a proof for completeness.

\begin{Lemma}\label{lmm:1}
Let $A \in \mathcal{CPSD}^n.$ The following chain of inequalities holds:
$$\cpsdrank_{\fG}(A) \leqslant \cpsdrank(A) \leqslant n\cdot\cpsdrank_{\fG}(A).$$
\end{Lemma}
\begin{proof} The first inequality is clear from the fact that the rank of a matrix is at most its size. For the second let 
 $A_1,\ldots, A_n \in \mathcal{S}_+^d$ be a Gram representation of $A \in \mathcal{CPSD}^n$ with  rank$(A_i)\leqslant \cpsdrank_{\fG}(A)$ for all $i\in[n]$.  Then the rank $r$ of the matrix  $$A^\prime \coloneqq \sum_{i=1}^{n}A_{i} \in \mathcal{S}_+^{d}$$ is at most $n\cdot\cpsdrank_{\fG}(A).$ By the spectral theorem we obtain  $$A^\prime =O\Diag(\lambda_{1}, \ldots, \lambda_{r}, 0, \ldots, 0)O^t=ODO^t$$ for some orthogonal matrix $O$ and $\lambda_{i} > 0$ for all $i \in [r]$. 
Now we have $$\langle O^tA_{i}O, O^tA_{j}O  \rangle= \text{Tr}(O^tA_jOO^tA_iO)={\rm Tr}(A_jA_i)=\langle A_{i}, A_{j} \rangle=A_{ij},$$
thus the matrices $O^tA_1O,\ldots, O^tA_nO \in \mathcal{S}_+^d$ form a Gram representation for $A$ as well. 
From the fact that  $O^tA_iO \succcurlyeq 0$ for all $i \in [n],$ and 
$$\sum_{i=1}^{n}O^tA_iO=O^tA^\prime O=\begin{pmatrix}
D_{r \times r} &0\\
0 & 0
\end{pmatrix}$$
it easily follows that each $O^tA_iO$ has nonzero entries only in the upper left $r\times r$-block as well. When restricting to this upper left block we obtain a Gram representation of $A$ with psd matrices of size $r$, which shows that $\cpsdrank(A) \leqslant r \leqslant n\cdot\cpsdrank_{\fG}(A).$
\end{proof}

As we have already seen above, the cpsd-rank is a natural non-commutative analogue of the cp-rank. However, while the cp-rank is upper bounded by a function that depends only on the matrix size,  there is no general such  upper bound on the cpsd-rank. There are only some classes of completely positive semidefinite matrices for which there exists an upper bound in terms of the matrix size. For instance, the authors in \cite{pra} and \cite{san} construct cpsd matrices of size $2n$ and $4n^2+2n+2$ for all $n \geqslant 1$ with cpsd-rank being $2^{\Omega(\sqrt{n})}$ and $2^n,$ respectively. This is impossible for $n > 10,$ since the cpsd-cone is not closed. This is the main motivation for the results in this paper. By the definition of the $\mathcal{CPSD}^n$-cone, going through larger and larger size of the psd matrices in a Gram representation is a procedure to produce all completely positive semidefinite matrices, and so we have 
$$\mathcal{CPSD}^n=\bigcup_{r \in \mathbb{N}}\mathcal{CPSD}^n_{\leqslant r},$$
where 
$$\mathcal{CPSD}^n_{\leqslant r} \coloneqq \{A=\left(\langle A_i,A_j\rangle\right)_{i,j} \ |\ A_1,\ldots, A_n\in \mathcal S^r_+ \}.$$ Note that we have $$\mathcal{CPSD}^n_{\leqslant r}\subseteq \mathcal{CPSD}^n_{\leqslant r+1}$$ for all $r$, since psd matrices of size $r$ can be enlarged to size $r+1$ without changing the inner product, by adding a zero row and column.

\begin{Lemma}\label{lmm:2}
For each $n,r \geqslant 1,$ the set $\mathcal{CPSD}^n_{\leqslant r}$ is closed and semialgebraic.
\end{Lemma}
\begin{proof} The closedness is very similar to the proof of the closedness of $\mathcal{CP}^n$ \cite{ber}.
Fix $n, r \geqslant 1.$ Let $\left(A^{(k)}\right)_{k\in \N}$ be a sequence of matrices in $\mathcal{CPSD}^n_{\leqslant r}$ converging to some $A \in \mathcal{S}^n,$ which clearly means 
$\lim_{k\to\infty}A_{ij}^{(k)}=A_{ij}$ for all $i,j \in [n].$ Now for each $k \in \N$ there exist $A_1^{(k)},\ldots, A_n^{(k)}\in \mathcal{S}_+^r$ such that $$A_{ij}^{(k)}=\text{tr}\left(A_i^{(k)}A_j^{(k)}\right).$$ The converging sequence $\left(A^{(k)}\right)_k$ is bounded, in particular, all the diagonal entries $$A^{(k)}_{ii}=\text{tr}\left(A_i^{(k)}A_i^{(k)}\right)=\left\Vert A_i^{(k)}\right\Vert^2$$ are bounded. So without loss of generality, by using the Bolzano-Weierstrass Theorem and the fact that $\mathcal{S}_+^{r}$ is  closed, we can assume that for each $i\in [n]$ the sequence $\left(A^{(k)}_{i}\right)_k$ converges to some $A_i \in \mathcal{S}_+^{r},$ and hence for each $i,j \in [n]$ we have
$$\text{tr}(A_{i}A_{j})=\lim_{k\to\infty}\text{tr}\left(A^{(k)}_{i}A^{(k)}_{j}\right)=\lim_{k\to\infty}A_{ij}^{(k)}=A_{ij}.$$
This shows that $A_1,\ldots, A_n\in \mathcal S^r_+$ form a Gram representation for $A$, and thus $A \in \mathcal{CPSD}^n_{\leqslant r}$. This proves closedness.

The membership of matrices in $\mathcal{CPSD}^n_{\leqslant r}$ can be stated as a first order formula in the language of ordered rings, using quantifiers. Indeed, the existence of the Gram representation is an existential formula, since  the size of the $A_i$ is bounded by $r$ (the quantification is over the entries of the $A_i$) and since positivity can be checked by the nonnegatity of the determinants of all principal minors of $A_i,$ which are polynomial inequalities. By quantifier elimination (see for example \cite{pd}) we conclude that the set $\mathcal{CPSD}^n_{\leqslant r}$ is indeed semialgebraic.
\end{proof}
\begin{Cor}
For $n \geqslant 10$ and for any $k$ there exists $M \in \mathcal{CPSD}^n$ such that $$\cpsdrank(M) \geqslant k.$$
\end{Cor}
\begin{proof} If the cpsd-rank  admitted a bound, there would exist some $r \geqslant 1$ with $\mathcal{CPSD}^n=\mathcal{CPSD}^n_{\leqslant r}.$ Consequently, by Lemma \ref{lmm:2}, the cone  $\mathcal{CPSD}^n$ would be closed, which was shown to fail in \cite{sl,dy} for $n \geqslant 10$. 
\end{proof}

To prove our main result, we make use of an approximate version of Carath\'eodory's Theorem \cite{adi} and a variant of the Johnson-Lindenstrauss Lemma \cite{ude}. For a set $P \subseteq \mathbb{R}^d$ we denote by $\textrm{conv}(P)$ the convex hull of $P$. Further, for $k \in \mathbb{N}$ we denote the set of all convex combinations from $P$ of length at most $k$ by $\textrm{conv}_{k}(P)$. The (exact) Carath\'eodory Theorem states that every element in a convex hull can be written as a convex combination of at most $d+1$ elements. Hence, we have the following increasing chain of sets:
$$P= \textrm{conv}_{1}(P)\subseteq \textrm{conv}_{2}(P)\subseteq \cdots\subseteq \textrm{conv}_{d+1}(P)=\textrm{conv}(P).$$
We now state an approximate version of Carath\'eodory's Theorem with respect to the $2$-norm \cite{adi}. Note that $\diam(P)$ denotes the diameter of the set $P$, i.e.\ the maximal distance between two points in $P$ with respect to the  Euclidean norm.

\begin{Thm}[Approximate Carath\'eodory Theorem]
\label{thm:approxCara2}
Let $P\subseteq \mathbb R^d$ be a bounded set and $\varepsilon > 0$. Then for
$$ k = \left\lceil \frac{\diam(P)^2}{2 \varepsilon^2}\right\rceil $$
the set ${\rm conv}_k(P)$ is $\varepsilon$-dense in ${\rm conv}(P)$, meaning that for each $a\in {\rm conv}(P)$  
there exists some $b \in {\rm conv}_{k}(P)$ such that $\Vert a - b \Vert_2 < \varepsilon$.
\end{Thm}
	
Since entrywise $2$-norm and Hilbert-Schmidt norm for matrices coincide,  this directly leads to the following rank approximation result for positive semidefinite matrices:
\begin{Cor}
\label{lmm:3} 
Let $A \in \mathcal{S}_+^d.$ Then  for every $\varepsilon >0$ there exists a positive semidefinite matrix $B \in \mathcal{S}_+^d$ such that  $${\rm tr}(B)={\rm tr}(A),$$
\[
\|A-B\|_{2}=\sqrt{{\rm tr}\left((A-B)^2\right)}< \varepsilon,\]
and 
\[{\rm rank}(B)\leqslant \left \lceil{\frac{{\rm tr}(A)^2}{\varepsilon^2}}\right \rceil. \]
\end{Cor}
\begin{proof}
Consider the set  $$P \coloneqq \left\{\textrm{tr}(A) \cdot u u^{t} : u \in \mathbb{R}^d, \Vert u \Vert_2 = 1\right\} \subseteq \mathcal{S}^{d}_+.$$
By the eigenvalue decomposition of $A$  it is immediate that  $A \in \textrm{conv}(P)$. Further, it is easy to check that  $$\diam(P) = \sqrt{2} \textrm{tr}(A),$$ and thus Theorem \ref{thm:approxCara2} implies the existence of some  $B$ with
$$ \sqrt{\textrm{tr}\left((A-B)^2\right)} = \Vert A - B \Vert_2 < \varepsilon$$
which is a convex combination of at most
$$ k = \left\lceil \frac{\textrm{tr}(A)^2}{\varepsilon^2}\right\rceil$$
elements from $P$. Since each element in $P$ is psd of rank $1$ its trace is equal to ${\rm tr}(A)$, this finishes the proof. 
\end{proof}
We will later also use the following version of the Johnson-Lindenstrauss Lemma \cite{ude}: 
\begin{Thm}[Johnson-Lindenstrauss Lemma]
\label{lmm:jl}
Let $0<\varepsilon<1,$ $\{x_1, \ldots, x_m\}\subseteq \mathbb{R}^d,$ and set $r \coloneqq \left\lceil 8\log(m+1)/\varepsilon^2\right\rceil.$ Then there exists a linear map $Q:\mathbb{R}^d \rightarrow \mathbb{R}^r$ such that 
\[ |x_i^{t}x_j-x_i^{t}Q^{t}Qx_j|\leqslant \varepsilon\left(\|x_i\|_2^2+\|x_j\|_2^2-x_i^{t}x_j\right) \quad \mbox{ for all } i,j \in  [m].
		\]
\end{Thm}

\section{Main Result}\label{Sec:Main}
We are now ready to state and prove our main result: 

\begin{Thm}\label{ThmMain}
Let $M=(\langle A_i, A_j\rangle)_{i,j=1}^{n} \in \mathcal{CPSD}^n$, set $\ell \coloneqq \max_{i}{\rm tr}(A_i)$ and $L \coloneqq \max_{i} M_{ii}$. Then for every $0<\varepsilon< \frac12\min \{\ell^2,  L\}$ there exists some  $N \in \mathcal{CPSD}^n$ with  
$$\cpsdrank(N)\leqslant \min \left\{n\left\lceil \frac{9L \ell^2}{2\varepsilon^2}\right\rceil, \frac{(6\ell)^4\log\left(n\left\lceil \frac{18 L \ell^2}{\varepsilon^2} \right\rceil+1\right)}{\varepsilon^2}\right\} $$ 
and
$$|M_{ij}-N_{ij}|<\varepsilon \quad \mbox{ for all } i,j\in [n].$$
\end{Thm}

\begin{proof}
  By Corollary \ref{lmm:3}, for every $i \in [n]$ there exists a psd matrix $A^{\prime}_i \in \mathcal{S}^{d}_+$
with ${\rm tr}(A_i')={\rm tr}(A_i)$  such that $$\Vert A_i - A^{\prime}_i\Vert_2 <\varepsilon_1 \coloneqq \sqrt{L} \left(\sqrt{1 + \frac{\varepsilon}{2L}} - 1 \right)$$
and
$$\textrm{rank}(A^{\prime}_i) \leqslant \left\lceil \frac{\ell^2}{\varepsilon_1^2} \right\rceil \leqslant \left\lceil \frac{18 L \ell^2 }{\varepsilon^2} \right\rceil$$
where we have used for the last inequality that
\begin{equation}\label{eq:sqrteq}
\sqrt{1+x} \leqslant 1 + \frac{x}{2} - \cfrac{x^2}{9} \quad \text{ for all } 0 \leqslant x \leqslant \frac{1}{4},
\end{equation}
applied to $x=\frac{\varepsilon}{2L}.$ We define  $M^{\prime} \coloneqq \left({\rm tr}\left(A^{\prime}_iA^{\prime}_j\right)\right)_{i,j=1}^{n}$
and show that $M^{\prime}$ is an $\varepsilon/2$-approximation of $M.$ Indeed,  for $i,j \in [n]$ we have
\begin{equation}\label{eq:1}
\begin{aligned}
\Big|\textrm{tr}(A_i A_j) - \textrm{tr}(A^{\prime}_i A^{\prime}_j)\Big| &\leqslant \Big| \textrm{tr}((A_i - A^{\prime}_i) A_j)\Big| +\Big|\textrm{tr}\left(A^{\prime}_i (A_j-A^{\prime}_j)\right)\Big| \\
&\leqslant \Vert A_i - A^{\prime}_i \Vert_2 \cdot \Vert A_j \Vert_2 + \Vert A_j - A^{\prime}_j \Vert_2 \cdot \Vert A^{\prime}_i \Vert_2\\
&\leqslant \Vert A_i - A^{\prime}_i \Vert_2 \cdot \Vert A_j \Vert_2 + \Vert A_j - A^{\prime}_j \Vert_2 \cdot \left(\Vert A^{\prime}_i - A_i \Vert_2 + \Vert A_i\Vert_2 \right)\\
&\leqslant 2\varepsilon_1 \sqrt{L} + \varepsilon_1^2 = \varepsilon/2,\end{aligned}
\end{equation}
where the first and third inequality follow from the triangle inequality and the second inequality is Cauchy-Schwarz. For the last inequality we have used $\Vert A_i\Vert_2 = \sqrt{M_{ii}} \leqslant \sqrt{L}.$
		
Replacing $\varepsilon$ by $2\varepsilon$ and using Lemma \ref{lmm:1} then establishes the first  of the upper bounds. We now continue to prove the second upper bound. For each $i \in [n]$ let the eigendecomposition of $A^{\prime}_i$ be
$$A^{\prime}_i=\sum_{k=1}^{m}\lambda_{k,i}u_{k,i}u_{k,i}^t,$$
where $m \coloneqq \left\lceil \ell^2/ \varepsilon_1^2 \right\rceil$, and consider the $nm$-point set of all (normalized) eigenvectors of all $A^{\prime}_i$: $$\bigcup_{i=1}^{n}\left\{u_{1,i}, \ldots, u_{m,i}\right\} \subseteq \mathbb{R}^d.$$ 
By applying Theorem \ref{lmm:jl} with $$\varepsilon_2 \coloneqq -\frac{1}{3}+\frac{1}{3}\sqrt{1+\frac{\varepsilon}{2\ell^2}}$$  we find that for  \begin{equation}\label{jlineq}r \leqslant \frac{(6\ell)^4\log\left(n\left\lceil \frac{18 L \ell^2}{\varepsilon^2} \right\rceil+1\right)}{\varepsilon^2}\end{equation} there is a linear map $Q:\mathbb{R}^d \rightarrow \mathbb{R}^r$ such that 
\begin{equation}\label{eq:2}
\left|u^{t}_{k,i}u_{k^\prime,j}-u^{t}_{k,i}Q^{t}Qu_{k^\prime,j}\right|\leqslant \varepsilon_2\left(2-u^{t}_{k,i}u_{k^\prime,j}\right) \leqslant 3\varepsilon_2
\end{equation}  for all $i,j \in [n], k, k^\prime \in [m].$ Note that for the inequality in \eqref{jlineq} we have again used \eqref{eq:sqrteq} with $x \coloneqq \varepsilon/2\ell^2$.

Set $v_{k,i} \coloneqq Qu_{k,i}\in \mathbb{R}^r$ for all $i \in [n]$ and $k \in [m].$ For each $i \in [n],$ define the new psd matrix $$A^{\prime\prime}_i \coloneqq \sum_{k=1}^{m}\lambda_{k,i} v_{k,i}v_{k,i}^{t} \in \mathcal{S}_{+}^{r},$$
and the new cpsd matrix $M^{\prime\prime}=\left(\text{tr}\left(A^{\prime\prime}_iA^{\prime\prime}_{j}\right)\right)_{i,j=1}^{n} \in \mathcal{CPSD}^n,$ whose cpsd-rank is at most $r.$ We finally check that  $M^{\prime\prime}$ is an $\varepsilon/2$-approximation of $M^{\prime}.$ For all $i,j \in [n]$ we have 
\begin{equation}\label{eq:4}
\begin{aligned}
\Big|\textrm{tr}\left(A^{\prime}_i A^{\prime}_j\right) - \textrm{tr}\left(A^{\prime\prime}_i A^{\prime\prime}_j\right)\Big| &\leqslant \sum_{k,k^\prime=1}^{m}\lambda_{k,i}\lambda_{k^\prime,j}\left|\left(u_{k,i}^{t}u_{k^\prime,j}\right)^2-\left(v_{k,i}^{t}v_{k^\prime,j}\right)^2\right| \\
&=\sum_{k,k^\prime=1}^{m}\lambda_{k,i}\lambda_{k^\prime,j}\left|u_{k,i}^{t}u_{k^\prime,j}-v_{k,i}^{t}v_{k^\prime,j}\right|\left|u_{k,i}^{t}u_{k^\prime,j}+v_{k,i}^{t}v_{k^\prime,j}\right| \\
&\leqslant \sum_{k,k^\prime=1}^{m}\lambda_{k,i}\lambda_{k^\prime,j} \left(6\varepsilon_2+9\varepsilon_2^2\right) \\
&= \left(6\varepsilon_2+9\varepsilon_2^2\right) \text{tr}(A^{\prime}_i)\text{tr}(A^{\prime}_j)\\
&\leqslant \left(6\varepsilon_2+9\varepsilon_2^2\right)  \ell^2=\varepsilon/2
\end{aligned}
\end{equation}
where the first inequality follows from the triangle inequality, and the third inequality from $\text{tr}(A^{\prime}_i)=\text{tr}(A_i)\leqslant \ell$. The second inequality follows from the inequality
$$\left|u_{k,i}^{t}u_{k^\prime,j}+v_{k,i}^{t}v_{k^\prime,j}\right| \leqslant 2 \left|u_{k,i}^{t}u_{k^\prime,j} \right| + \left|u_{k,i}^{t}u_{k^\prime,j}-v_{k,i}^{t}v_{k^\prime,j}\right| \leqslant 2 + \left|u_{k,i}^{t}u_{k^\prime,j}-v_{k,i}^{t}v_{k^\prime,j}\right| $$
where we have used that $u_{k,i}$ is normalized, and from Equation $\eqref{eq:2}$.

Altogether, the matrix $N \coloneqq M^{\prime\prime}$ is an $\varepsilon$-approximation of  $M,$ whose  cpsd-rank is small enough to verify the second upper bound.
\end{proof}

Let us comment on the assumptions and conclusions of our main theorem.  First, which of the bounds in the main theorem is better depends on our setup. For instance, if we fix  $\varepsilon$ and let $n$ approach infinity, then the bound obtained by applying the Johnson-Lindenstrauss Lemma is significantly smaller than the other. On the other hand, for $n$ fixed and $\varepsilon$ getting smaller, the first bound will be  better.

Second, if $M = (\langle A_i, A_j \rangle)_{i,j=1}^{n}$ is a completely positive matrix, there are diagonal matrices $A_1, \ldots, A_n \in \mathcal{S}^{d}_{+}$ providing a cpsd decomposition. The first approximation procedure can be used to generate a completely positive approximation. Indeed applying the Approximate Carath\'eodory Theorem will return again diagonal matrices $A_i'$ of a smaller size. Since a cp decomposition is equivalent to a cpsd decomposition with diagonal matrices, the approximation $N$ is completely positive with
$$\cprank(N) \leqslant n \left\lceil \frac{9L \ell^2}{2 \varepsilon^2} \right\rceil.$$

Finally,  note that the number $\ell$ in Theorem \ref{ThmMain} is a kind of hidden complexity measure of the cpsd-matrix $M$. What we get from $M$ directly are the numbers ${\rm tr}\left(A_i^2\right)=M_{ii}$, for any Gram representation $A_1,\ldots, A_n\in\mathcal S^d_+$. The numbers ${\rm tr}(A_i)$ are not uniquely determined however, and  they encode information about the eigenvalue distribution of the psd-matrices in a Gram representation. One could upper bound them in terms of the numbers ${\rm tr}\left(A_i^2\right)=M_{ii}$, but this would involve a constant depending on $d$ in general, which we want to avoid. So one should not employ this upper bound, and understand the approximation to really depend on the hidden complexity of $M$, but not on its cpsd-rank. One instance where this works well is when $M$ admits a Gram decomposition $A_1,\ldots, A_n$ with all nonzero eigenvalues of all $A_i$ larger equal to $1$. A special case is stated in the following corollary.

\begin{Cor}\label{Cor}
Let $M \in \mathcal{CPSD}^n$ with Gram representation consisting of orthogonal projections $P_1,\ldots, P_n\in  \mathcal{S}_+^{d}.$ Further set $L \coloneqq \max_iM_{ii}.$ Then for all $0<\varepsilon<\frac{1}{2}\L^2$ there exists some  $N \in \mathcal{CPSD}^n$ with 
$$\cpsdrank(N)\leqslant \min \left\{n\left\lceil \frac{9L^3}{2\varepsilon^2}\right\rceil, \frac{(6L)^4\log\left(n\left\lceil \frac{18L^3}{\varepsilon^2} \right\rceil+1\right)}{\varepsilon^2}\right\} $$
and 
$$\left|M_{ij}-N_{ij}\right|<\varepsilon \quad \mbox{ for all } i,j\in [n].$$
\end{Cor}
\begin{proof}
This is immediate from Theorem \ref{ThmMain}, since for orthogonal projections 
we have $\text{tr}\left(P_i\right)=\text{tr}\left(P_i^2\right)=M_{ii},$ and thus $L=\ell$.
\end{proof}

\begin{Example}   Consider the identity matrix $I_n\in \mathcal{CPSD}^n.$ A Gram representation is given by the elementary matrices $A_i \coloneqq E_{ii}\in \mathcal{S}_+^{n},$ and it is not hard to check that there is no Gram representation of smaller size, i.e.\ $\cpsdrank(I_n)=n$. The given $A_i$ are rank one projections, so we have $\ell=\max_i {\rm tr}(A_i)=1=L$ for the given representation. The first upper bound from Theorem \ref{ThmMain}/Corollary \ref{Cor} is not meaningful here, but the second is  $$\frac{6^4\log\left(n\left\lceil \frac{18}{\varepsilon^2} \right\rceil+1\right)}{\varepsilon^2},$$ which is smaller than $n$ for fixed $\varepsilon$ and large enough $n$. For example, for  $\varepsilon=1/2$ this happens at around $n=8\times 10^4,$ for $\varepsilon=1/10$ at around $n=2.9\times 10^6$.
\end{Example}

\begin{Example}
 Let $M\in\mathcal{CP}^n$ be a completely positive matrix. If we assume  all diagonal entries of $M$ to be  one, this means $M$ has a Gram representation by nonnegative unit vectors $v_1,\ldots, v_n$ (of some dimension). If we further understand these nonnegative vectors as diagonals of psd-matrices, we obtain a cpsd-decomposition of $M$  for which the constant $\ell$ from Theorem \ref{ThmMain} is precisely the maximum over all $1$-norms of the $v_i$. The first upper bound for the approximation thus becomes 
$$n\left\lceil \frac{9 \max_i \Vert v_i\Vert_1^2}{2\varepsilon^2}\right\rceil$$
which, depending on the $v_i$, might be much smaller than the only known upper bound to the cp/cpsd-rank so far, which is the actual cp-rank of $M$ and bounded by  ${n+1\choose 2}-4$ (see \cite{sb2}). Note again that the resulting approximation will again be completely positive.
\end{Example}

\begin{Example}
Let $a,b,c,d \in \mathbb{R}_{+}^n$ with strictly positive entries. Further, define $C \coloneqq \Diag(c_1, \ldots, c_n) \in \mathcal{S}^{n}_{+}$ and $D \coloneqq \Diag(d_1, \ldots, d_n) \in \mathcal{S}^{n}_{+}$. Then, by Proposition 2.1 in \cite{sb} the $n^2 \times 2n$ matrix
$$ V = \big(b \otimes C \: \big| \: D \otimes a \big)$$
generates a completely positive matrix $M = V^{t} V \in \mathcal{CP}^{2n}$ with $\cprank(M) = n^2$.

Now for $q \in (0,1)$ set $c,d \coloneqq (1,1,\ldots, 1)^t$, and $a,b \coloneqq (1-q) \cdot \left(1,q,q^2,q^3, \ldots, q^{n-1}\right)^t$. For the columns $v_i$ of $V$ it holds that
$$\Vert v_i \Vert_1 = (1-q) \cdot \sum_{k=0}^{n-1} q^k \leqslant 1.$$
Thus both $\ell$ and $L$ from Theorem \ref{ThmMain} are at most $1$, and hence, by the observation in ($i$), there exists a completely positive matrix $N \in \mathcal{CP}^{2n},$ which is an $\varepsilon$-approximation of $M$, and fulfills
$$\cpsdrank(N) \leqslant \cprank(N) \leqslant 2n \left\lceil \frac{9}{2\varepsilon^2} \right\rceil.$$
\end{Example}


\section*{Acknowledgements} The first author gratefully acknowledges funding by the Austrian Academy of Sciences (\"OAW), through the DOC-Scholarship 25171. The second author gratefully acknowledges funding by the Austrian Science Fund (FWF), through project P 33122-N and funding of the Austrian Academy of Sciences (\"OAW) through the DOC scholarship 26547. We thank Werner Schachinger for a helpful comment on a previous version of this paper.


\newpage
\begin{bibdiv}
\begin{biblist}	
\bib{adi}{article}{
AUTHOR = {Adiprasito, K.},
AUTHOR={ B{\'{a}}r{\'{a}}ny, I.},
AUTHOR={Mustafa, N. H.},
TITLE = {Theorems of Carath{\'{e}}odory, Helly, and Tverberg without dimension},
JOURNAL = {Proc. Annu. ACM-SIAM Sympos. Discrete Algorithms},
YEAR = {2019},
PAGES = {2350--2360},
}
\bib{ber}{book}{
AUTHOR = {Berman, A.},
AUTHOR={Shaked-Monderer, N.},
TITLE = {Completely Positive Matrices},
PUBLISHER={World Scientific},
YEAR = {2003},
}

\bib{bom}{article}{
AUTHOR ={Bomze, I. M.},
AUTHOR={Schachinger, W.},
AUTHOR={Ullrich, R.},
TITLE = {New Lower Bounds and Asymptotics for the cp-Rank},
JOURNAL = {SIAM J. Matrix Anal. Appl.},
VOLUME = {36},
Number={1},
YEAR = {2015},
PAGES = {20--37},
}
\bib{bu}{article}{
AUTHOR = {Burgdorf, S.},
AUTHOR={Laurent, M.},
AUTHOR={Piovesan, T.},
TITLE = {On the closure of the completely positive semidefinite cone
	and linear approximations to quantum colorings},
JOURNAL = {Electron. J. Linear Algebra},
FJOURNAL = {Electronic Journal of Linear Algebra},
VOLUME = {32},
YEAR = {2017},
PAGES = {15--40},
}
\bib{dy}{article}{
AUTHOR={Dykema, K.},
AUTHOR={Paulsen, V. I.},
AUTHOR={Prakash, J.},
TITLE={Non-closure of the set of quantum correlations via graphs},
JOURNAL={Commun. Math. Phys.},
VOLUME={365},
YEAR={2019},
}
\bib{faw}{article}{
AUTHOR = {Fawzi, H.},
AUTHOR={Gouveia, J.},
AUTHOR={Parrilo, P. A.},
AUTHOR={Robinson, R.},
AUTHOR={Thomas, R. R.},
TITLE = {Positive semidefinite rank},
JOURNAL = {Math. Program.},
FJOURNAL = {Mathematical Programming},
VOLUME = {153},
YEAR = {2015},
NUMBER = {1, Ser. B},
}
\bib{fre}{article}{
AUTHOR = {Frenkel, P. E.},
AUTHOR={Weiner, M.},
TITLE = {On vector configurations that can be realized in the cone of positive matrices},
JOURNAL = {Linear Algebra Appl.},
FJOURNAL = {Linear Algebra and its Applications},
VOLUME = {459},
YEAR = {2014},
PAGES = {465--474},
}
\bib{san}{article}{
AUTHOR={Gribling, S.},
AUTHOR={de Laat, D.},
AUTHOR={Laurent, M.},
TITLE={Matrices with high completely positive semidefinite rank},
JOURNAL={Linear Algebra Appl.},
VOLUME={513},
YEAR={2017},
PAGES={122-148},
}

\bib{lau}{article}{
AUTHOR={Laurent, M.},
AUTHOR={Piovesan, T.},
TITLE={Conic approach to quantum graph parameters using linear
		optimization over the completely positive semidefinite cone},
JOURNAL={SIAM J. Optim.}, 
NUMBER={25},
YEAR={2015},
PAGES={2461--2493},
}
\bib{max}{article}{
AUTHOR={Maxfield, J. E.},
AUTHOR={Minc, H.},
TITLE={On the matrix equation $X'X=A$ },
JOURNAL={P. Edinburgh Math. Soc.}, 
VOLUME={13},
NUMBER={2},
YEAR={1962},
PAGES={125 --129},
}
\bib{pra}{article}{
AUTHOR={Prakash, A.},
AUTHOR={Sikora, J.},
AUTHOR={Varvitsiotis, A.},
AUTHOR={Wei, Z.},
TITLE={Completely positive semidefinite rank},
VOLUME={171},
JOURNAL={Math. Program.},
YEAR={2017},
NUMBER={1-2},
PAGES={397–431},
}

\bib{pd}{book}{
    AUTHOR = {Prestel, A.}
    AUTHOR={Delzell, Ch. N.},
     TITLE = {Positive polynomials},
    SERIES = {Springer Monogr. Math.},
      NOTE = {From Hilbert's 17th problem to real algebra},
 PUBLISHER = {Springer-Verlag, Berlin},
      YEAR = {2001},
     PAGES = {viii+267},
}


\bib{sch}{article} {
    AUTHOR = {Schachinger, W.},
     TITLE = {Lower bounds for maximal cp-ranks of completely positive
              matrices and tensors},
   JOURNAL = {Electron. J. Linear Algebra},
  FJOURNAL = {Electronic Journal of Linear Algebra},
    VOLUME = {36},
      YEAR = {2020},
     PAGES = {519--541},
}


\bib{sb}{article}{
author = {Shaked-Monderer, N.},
author = {Bomze, I.},
author = {Jarre, F.},
author = {Schachinger, W.},
year = {2013},
month = {04},
pages = {355-368},
title = {On the cp-Rank and Minimal cp Factorizations of a Completely Positive Matrix},
volume = {34},
journal = {SIAM J. Matrix Anal. Appl.},
doi = {10.1137/120885759}
}


\bib{sb2}{article}{
    AUTHOR = {Shaked-Monderer, N.},
    author={Berman, A.},
    author={Bomze, I.},
    author={Jarre, F.},
    author={Schachinger, W.},
     TITLE = {New results on the cp-rank and related properties of
              co(mpletely) positive matrices},
   JOURNAL = {Linear Multilinear Algebra},
  FJOURNAL = {Linear and Multilinear Algebra},
    VOLUME = {63},
      YEAR = {2015},
    NUMBER = {2},
     PAGES = {384--396},
 }

\bib{si}{article}{
	AUTHOR = {Sikora, J.},
	author = {Varvitsiotis, A.},
	TITLE = {Linear conic formulations for two-party correlations and values of nonlocal games},
	JOURNAL = {Mathematical Programming},
	VOLUME = {162},
	YEAR = {2017},
	PAGES = {431--463},
}

\bib{sl}{article}{
AUTHOR={Slofstra, W.},
TITLE={The set of quantum correlations is not closed},
JOURNAL={Forum Math. Pi},
VOLUME={7},
YEAR={2019},
PAGES={E1},
}
\bib{ude}{article}{
AUTHOR={Udell, M.},
AUTHOR={Townsend, A.},
TITLE={Why are Big Data Matrices Approximately Low Rank?},
JOURNAL={SIAM J. Math. Data Sci.},
VOLUME={1},
YEAR={2019},
PAGES={144-160},
}





\end{biblist}
\end{bibdiv}
\end{document}